\documentclass[12pt, twoside]{article}
\usepackage{amsmath,amsthm,amssymb}
\usepackage{times}
\usepackage{enumerate}
\usepackage{bm}
\usepackage[all]{xy}
\usepackage{mathrsfs}
\usepackage{amscd}
\usepackage{color}
\def\titlerunning#1{\gdef\titrun{#1}}
\makeatletter
\def\author#1{\gdef\autrun{\def\and{\unskip, }#1}\gdef\@author{#1}}
\def\address#1{{\def\and{\\\hspace*{18pt}}\renewcommand{\thefootnote}{}%
		\footnote {#1}}%
	\markboth{\autrun}{\titrun}}
\makeatother
\def\email#1{e-mail: #1}
\def\subjclass#1{{\renewcommand{\thefootnote}{}%
		\footnote{\emph{Mathematics Subject Classification (2020):} #1}}}
\def\keywords#1{\par\medskip
	\noindent\textbf{Keywords.} #1}

\newtheorem{theorem}{Theorem}[section]
\newtheorem{corollary}[theorem]{Corollary}
\newtheorem{lemma}[theorem]{Lemma}
\newtheorem{proposition}[theorem]{Proposition}
\theoremstyle{definition}
\newtheorem{definition}[theorem]{Definition}
\newtheorem{remark}[theorem]{Remark}

\numberwithin{equation}{section}

\xyoption{all}

\def \a {\alpha }
\def \b {\beta}

\def \de {\delta}
\def \De {\Delta}
\def \la {\lambda}
\def \La {\Lambda}

\def\Om{\Omega}

\def\na {\nabla}
\def\Ga{\Gamma}

\begin{document}
	\baselineskip=17pt
	
	\titlerunning{Eigenvalue Estimate for the Rough Laplacian on $1$-Forms and its Applications}
	\title{Eigenvalue Estimate for the Rough Laplacian on $1$-Forms and its Applications
	}
	
	\author{Teng Huang and Weiwei Wang}
	
	\date{}
	
	\maketitle
	
\address{Teng Huang: School of Mathematical Sciences, CAS Key Laboratory of Wu Wen-Tsun Mathematics, University of Science and Technology of China, Hefei, Anhui, 230026, People’s Republic of China; \email{htmath@ustc.edu.cn;htustc@gmail.com}}
\address{Weiwei Wang: School of Mathematical Sciences, University of Electronic Science and Technology of China, Chengdu, Sichuan, 611731, People’s Republic of China; \email{wawnwg123@163.com;wangweiwei123@uestc.edu.cn}}

\begin{abstract}
In this article, we establish a geometric lower bound for the first positive eigenvalue $\la^{(1)}_{1}$ of the rough Laplacian acting on $1$-forms for closed $2n$-dimensional Riemannian manifolds with nonvanishing Euler characteristic. In contrast to the case of functions, such a Li-Yau-type estimate does not hold in general, as evidenced by existing counterexamples.  Under assumptions including a lower bound on Ricci curvature, an upper bound on diameter, and an $L^{2p}$-norm bound on the Riemann curvature tensor, we prove that $\la^{(1)}_{1}$ is bounded below by a positive constant depending on these parameters. As applications, we derive vanishing results for the Euler characteristic under certain Ricci curvature bounds and the presence of a nonzero Killing vector field, extending classical Bochner-type theorems. 
\end{abstract}

\section{Introduction}
Let $(X,g)$ be a closed and connected Riemannian manifold and $\na$ the Levi-Civita connection associated to the metric $g$. We consider the \textit{rough Laplacian} 
$$\overline{\De}=\na^{\ast}\na$$
 acting on differential $p$-forms. This is a second-order elliptic operator whose spectrum consists of an unbounded sequence of real eigenvalues $\{\la^{(p)}_{k} \}_{k\in\mathbb{N}}$ that can be arranged in increasing order as
$$0=\la^{(p)}_{0}<\la^{(p)}_{1}\leq\la^{(p)}_{2}\leq\cdots\leq\la^{(p)}_{k}\leq\la^{(p)}_{k+1}\leq\cdots\nearrow+\infty.$$
Here, $\la^{(p)}_{0}$ is the zero eigenvalue whose its multiplicity equals $\dim\ker\na$. If $\dim\ker\na=0$, or equivalently if $(X,g)$ admits no nonzero parallel $p$-forms, then $\la^{(p)}_{0}$ does not exist and the spectrum conventionally begins with $\la^{(p)}_{1}>0$. When $0$ is not in the spectrum, every nonzero $p$-form $\a$ on $X$ satisfies the Rayleigh quotient inequality
$$R(\a)=\frac{\int_{X}|\na\a|^{2} }{\int_{X}|\a|^{2}}\geq\la^{(p)}_{1}.$$

Through the associated Weitzenb\"{o}ck formulas, the rough Laplacian is intimately related to other Laplace-type operators, such as the Dirac Laplacian, the Hodge Laplacian on $p$-forms, and Schr\"{o}dinger operators  (cf.  \cite{Bes,Pet}). Understanding the eigenvalues of the rough Laplacian is crucial in geometry and mathematical physics. There are many related results, see for instance  \cite{Bal,Ber,Col,CC,CM,EKU,GM,Man}. 

We recall some classical results for the smallest positive eigenvalue $\la_{1}^{(0)}$ on functions (i.e. $0$-forms), whose geometry is well studied. Cheeger \cite{Che70} and Buser \cite{Bus} showed that $\la_{1}^{(0)}$ is comparable to the square of the Cheeger isoperimetric constant of $X$. Yau provided more easily computable geometric quantities, such as diameter, volume, and lower bounds of Ricci curvature, to estimate the lower bound of $\la_{1}^{(0)}$. Li and Yau further developed a gradient estimate technique for the first eigenfunction, leading to an effective lower bound for $\la^{(0)}_{1}$. Their result can be stated as follows:

\begin{theorem}(cf. \cite{LY})
Let $(X,g)$ be a closed $n$-dimensional smooth Riemannian manifold. Suppose that the Ricci curvature satisfies 
$$\operatorname{Ric}(g)\geq -(n-1)\kappa,\ and\ {\rm{diam}}X\leq D(g),$$
then there exists a positive constant $c(n)$ such that the eigenvalue of rough Laplacian of $0$-form obeys
\begin{equation}\label{LY1}
\la^{(0)}_{1}D^{2}(g)\geq c^{-1}(n)\exp\{-[1+(1+2c^{2}(n)\La^{2})^{\frac{1}{2} }]\}.
\end{equation}	
\end{theorem}
We are interested in the lower bounds of the smallest positive eigenvalue  $\la_{1}^{(q)}$ of rough Laplacian on $q$-forms with $0<q<n$. By Colbois and Maerten \cite{CM}, there exists no positive universal lower bound for $\la_{1}^{(q)}$ under fixed volume. Furthermore, for each  $k\geq1$, Ann\'{e} and Takahashi (see \cite{AT24,AT25}) constructed on any closed manifold a family of Riemannian metrics, with fixed volume such that the $k$th positive eigenvalue $\la_{k}^{(q)}$ of the rough or the Hodge Laplacian acting on differential $q$-forms converges to zero.

Based on these observations, to obtain the lower bounds,  some geometric or topological conditions are required. In \cite{Bal}, Ballmann, Br\"{u}ning and Carron obtained a lower bound for manifolds with restrictions on holonomy. Lipnowski and Stern in \cite{LS} considered geometric estimates of the smallest positive eigenvalue of the $1$-form Laplacian on a closed hyperbolic manifold.
More recently, Boulanger and Courtois in \cite{BC} established a Cheeger-type inequality for coexact $1$-forms on closed orientable Riemannian manifolds.

A Riemannian manifold $X$ is said to have almost nonnegative sectional curvature if it admits a sequence of Riemannian metrics $g_{i}$ such that 
$$\operatorname{sec}(g_{i})\geq-\frac{1}{i},\ and\  D(g_{i})\leq1,$$
where $\operatorname{sec}(g_{i})$ denotes the sectional of $g_{i}$ and the $D(g_{i})$  the diameter of $g_{i}$. For a closed Riemannian manifold with almost nonnegative sectional curvature and nonzero first de Rham cohomology group, the Euler characteristic must vanish.  This was shown by Chen in \cite{Che} using Morse-Novikov  cohomology and by Yamaguchi in \cite{Yam} via collapsing theory.  

If a sequence of Riemannian metrics $\{g_{i}\}_{n\in\mathbb{N}}$ on a smooth manifold $X$ satisfies
$$\operatorname{Ric}(g_{i})\geq-\frac{1}{i},\ and\  D(g_{i})\leq1$$
for all $i\in\mathbb{N}$, then the Riemannian manifold is said to have almost nonnegative Ricci curvature. Here $\operatorname{Ric}(g_{i})$ stands for the Ricci curvature of $g_{i}$. For a closed Riemannian manifold $(X,g_{i})$ with almost nonnegative Ricci curvature and nonzero first de Rham cohomology group, Theorem 1 in \cite{Yam} implies that the Euler number vanishes provided the sectional curvature of $g_{i}$ has a uniform upper bound. Chen in \cite{Che} reproved this result using  Morse-Novikov cohomology. Moreover, if the metrics in the sequence above are K\"{a}hler, and the manifold has infinite fundamental group, then Chen, Ge and Han in \cite{CJH} proved that the Todd genus also vanishes. 
	
Let $\mathcal{M}(n,D)$ denote the family of closed Riemannian $n$-manifolds $(X,g)$ with sectional curvature $\operatorname{sec}(g)\leq1$ and diameter $D(g)\leq D$. Yamaguchi in \cite{Yam88} proved that there is an $\varepsilon> 0$ depending on $n$ and $D$ such that if $X\in\mathcal{M}(n,D)$ satisfy $\operatorname{Ric}(g)>-\varepsilon$, then $X$ is a fiber bundle over a $b_{1}(X)$-torus. In particular, if $b_{1}(X)=n-1$, then $X$ is diffeomorphic to an infranilmanifold,and if $b_{1}(X)=n$, then $X$ is diffeomorphictoan $n$-torus.  The following conjecture due to Gromov (\cite[5.22]{Gro80} or \cite[5.22]{Gro07}).\\
(\textbf{Gromov's conjecture}) There is an $\varepsilon>0$ depending on $n$ and $D$, such that if $X$ is an $n$-dimensional with $b_{1}(X)=n$ and $$\operatorname{Ric}(g)\geq-\varepsilon,\ and\ {\rm{diam}}(X)\leq D,$$ 
then $X$ is a homeomorphic $n$-torus.

This conjecture has been proved by Colding, see \cite{Col96}. Anderson \cite{An94} constructed numerous manifolds with $b_1 \leq n-1$ that are not fibered over $T^{b_1}$, yet admit metrics with arbitrarily small Ricci curvature. Consequently, this conjecture fails in general.  

We will use the manifolds constructed by Anderson to demonstrate that the eigenvalue $\la^{(1)}_{1}$ lower bound estimate also generally fails to satisfy the Li-Yau inequality (see Definition \ref{D1}).
\begin{proposition}\label{P2}
There exist closed Riemannian manifolds such that the eigenvalues of the rough Laplacian on $1$-forms do not satisfy a Li-Yau-type inequality.	
\end{proposition}

In this article, under Ricci curvature and Riemannian curvature conditions, we establish geometric lower bounds for $\lambda_1^{(1)}(X)$ on a closed $2n$-dimensional Riemannian manifold $(X,g)$ with nonvanishing Euler characteristic $\chi(X) \neq 0$.

\begin{theorem}\label{T3}

	Let $(X,g)$ be a closed $2n$-dimensional smooth Riemannian manifold with nonvanishing Euler characteristic. Suppose the Riemannian curvature satisfies
	\[
	\operatorname{Ric}(g) \geq -(2n-1)\kappa, \quad \text{and} \quad \operatorname{diam}(X) \leq D(g).
	\]
	Then the first eigenvalue of the rough Laplacian on $1$-forms satisfies
	\[
	\sqrt{\lambda^{(1)}_1} D \geq \min\left\{ 
	\left( \frac{\tilde{C}(n,p)}{1 + \sqrt{\|\mathrm{Riem}\|_{2p}D^{2}} }e^{-(2n-1)\sqrt{\kappa D^2}} \right)^{\frac{2pn}{p-n}},
	e^{-(2n-1)\sqrt{\kappa D^2}}
	\right\},
	\]
	where $\tilde{C}(n,p)$ is a positive constant and $p>n$.
\end{theorem}

\begin{remark}
The spectrum of $\De_{d}$ acting on $p$-forms will be denoted by
$$0=\mu^{(p)}_{0}<\mu^{(p)}_{1}\leq\mu^{(p)}_{2}\leq\cdots\leq\mu^{(p)}_{k}\leq\mu^{(p)}_{k+1}\leq\cdots\nearrow+\infty.$$
Colbois and Courtois in \cite[Theorem 0.4]{CC} proved that if a closed $n$-dimensional manifold $(X,g)$ satisfies 
$$|{\rm{sec}}_{g}|\leq K,\ {\rm{diam}}(X)\leq D, \ and\ {\rm{Vol}}(g)\geq\nu,$$
then there exists a constant $C(n,p,K,D,\nu) > 0$ such that for every $p = 1, \dots, n-1$,
	\[
	\mu^{(p)}_1\geq C(n, p,K,D, V).
	\]
The main idea of their proof is based on the fact that the space of manifolds satisfying the conditions of their theorem is compact. Although we can use the same reasoning to show that $\la^{(p)}_{0}$ also has a positive lower bound, we cannot obtain an explicit expression for it.	Suppose that $(X, g)$ is a closed $2n$- manifold of with Euler characteristic $\chi(X) \neq 0$ and the sectional curvature is bounded by \(K > 0\), i.e., $|\sec_{g}| \leq K$, then by Gauss-Bonnet-Chern formula, the volume of the manifold satisfies the following lower bound
\[
{\rm{Vol}}(g) \geq \frac{|\chi(X)|}{c(n) K^{n}}
\]
where $c(n)$ is a positive constant. Our Theorem \ref{T3} provides a specific explicit lower bound for the eigenvalue $\la^{(1)}_{1}$ for manifolds satisfying the above conditions.
\end{remark}
For a smooth Riemannian manifold $(X,g)$, let $\operatorname{Iso}(X,g)$ denote the group of Riemannian isometries $f : (X,g) \to (X,g)$. The Lie algebra of $\operatorname{Iso}(X,g)$ is spanned by the Killing vector fields on $(X,g)$. A classical theorem of Bochner states that the isometry group of a compact Riemannian manifold with negative Ricci curvature is finite (see \cite{Boc}). Several authors have since attempted to estimate the order of $\operatorname{Iso}(X,g)$ under negative Ricci curvature assumptions (see \cite{DSW,Kat,KK}).

\begin{corollary}\label{C4}
	Let $(X,g)$ be a closed $2n$-dimensional smooth Riemannian manifold with nonvanishing Euler characteristic. Suppose
	\[
	-(2n-1)\kappa \leq \operatorname{Ric}(g) < \lambda, \quad
	\|\mathrm{Riem}\|_{2p} \leq K, \quad
	\operatorname{diam}(X) \leq D,
	\]
	where
$$\sqrt{\lambda D^2}= \min\left\{ 
		\left( \frac{\tilde{C}(n,p)}{1 + \sqrt{KD^{2}} }e^{-(2n-1)\sqrt{\kappa D^2}} \right)^{\frac{2pn}{p-n}},
		\tilde{C}(n,p)e^{-(2n-1)\sqrt{\kappa D^2}}
		\right\},\ and\ p>n.$$ 
Then the isometry group of $X$ is finite.
\end{corollary}
\begin{remark}
Chen and Han \cite{CH} proved that if $X$ is a closed $2n$-dimensional complex manifold with nonzero holomorphic Euler characteristic, then for any fixed $\lambda_1 > 0$, there exists $\varepsilon > 0$ depending on $\lambda_1$ and $n$ such that for any K\"{a}hler metric $g$ on $X$, the isometry group of $(X, g)$ is finite, provided that
\[
-\lambda_1 \leq \operatorname{Ric}(g) \leq \varepsilon \quad \text{and} \quad D(g) \leq 1.
\]
A similar conclusion holds when $X$ is a closed $n$-manifold with nonzero Euler characteristic or nonzero signature, or a closed $4n$-dimensional spin manifold with nonzero elliptic genera, under the same Ricci curvature and diameter assumptions, provided the Riemann curvature tensor also satisfies an $L^p$-bound:
\[
\|\mathrm{Riem}\|_p \leq \lambda_2,
\]
for some $\lambda_2 > 0$, where $2p > n$. Their argument relied on a vanishing theorem for the Dirac operator coupled to a Killing vector field under the given curvature conditions, which provided an obstruction to the existence of a topological invariant. However, their derivation did not explicitly impose an upper bound on the Ricci curvature. If one could show that the Killing vector field is everywhere nonvanishing under these curvature conditions, then by Bott's results \cite{Bot}, it would follow that all Pontryagin numbers of such a manifold must vanish.	

\end{remark}
\noindent\textbf{Notations}\\
Throughout this article, $E\rightarrow X$ will denote a smooth vector bundle over a closed $n$-dimensional Riemannian manifold $(X,g)$. We denote by $\mathrm{Riem}$ and $R^{E}$ the curvature tensors of $X$ and $E$, and ${\rm{Ric}}$ respectively the Ricci curvature of $X$. We also define for any function $f:X\rightarrow\mathbb{R}$ its positive $$f^{+}(x)=\sup\{0,f(x) \}=\frac{f(x)+|f(x)| }{2}$$ and its negative part $$f^{-}(x)=\sup\{0,-f(x)\}=\frac{|f(x)|-f(x)}{2}.$$
The normalized $L^{p}$-norms are defined by
$$\|f\|_{p}=\bigg{(} \frac{1}{{\rm{Vol}}(g)}\int_{X}|f|^{p}dvol_{g}\bigg{)}^{\frac{1}{p}}.$$

\section{Preliminaries}

\subsection{Spectral Theory of the Rough Laplacian}
In this section, we recall basic facts on the rough Laplacian. (For a general reference, see Appendix of \cite{CM}  or \cite{Man}). Let $(X, g)$ be a closed connected $n$-dimensional Riemannian manifold without boundary, and let $(E,\na)$ be a Riemannian vector bundle with finite-dimensional fiber over $X$--that is, $E$ is a vector bundle equipped with a smooth metric $\langle \cdot, \cdot \rangle$ and a compatible connection $\nabla$. On the space $\Gamma(E)$ of smooth sections of $E$, we define the $L^2$-inner product $(\cdot, \cdot)$ induced by $\langle \cdot, \cdot \rangle$ and $g$. The connection $\nabla$ extends naturally to $p$-tensors on $X$ taking values in $E$. Its formal adjoint $\nabla^*$ with respect to the $L^2$-inner product defines the rough Laplacian (or connection Laplacian) acting on $\Gamma(E)$ by
\[
\overline{\Delta} = \nabla^* \nabla.
\]
The spectrum of $\overline{\De}$ is discrete and nonnegative, and can be written as
$$0=\la_{0}(E)<\la_{1}(E)\leq\la_{2}(E)\leq\cdots\leq\la_{k}(E)\leq\la_{k+1}(E)\leq\cdots\nearrow+\infty.$$
For a nonzero section $s \in \Gamma(E)$, the Rayleigh quotient is given by 
$$R(s)=\frac{\|\na s\|^{2}_{L^{2}(X)} }{\|s\|^{2}_{L^{2}(X)} }.$$
Let us define $E_{0}:=\ker\na$, which is a finite dimensional vectorial subspace of $\Ga(E)$. Every section $\alpha \in E_0$ is an eigenfunction corresponding to the zero eigenvalue. Define the orthogonal complement
$$H_{0}=\{\a\in\Ga(E)|\forall\a_{0}\in E_{0},(\a,a_{0})=0 \}.$$
Then the first positive eigenvalue satisfies
$$\la_{1}(E)=\inf\{R(\a)|\a\in H_{0},\|\a\|_{L^{2}(X)}\neq0 \}.$$

A classical theorem of Hopf states that the Euler characteristic of a closed Riemannian manifold equals the signed count of zeros of a generic vector field. It follows that if a closed Riemannian manifold admits a nonzero parallel vector field, then its Euler characteristic must vanish.
\begin{lemma}\label{L5}
Let $(X,g)$ be a closed $2n$-dimensional smooth Riemannian manifold with nonvanishing Euler characteristic $\chi(X)\neq 0$. Then 
$$\la^{(1)}_{1}=\inf\{R(\a)|\a\in\Om^{1}(X),\|\a\|_{L^{2}(X)}\neq0 \}.$$
\end{lemma}
\begin{proof}
Since  $\chi(X)\neq 0$, the manifold admits no nontrivial parallel $1$-forms--otherwise $\chi(X)$ would be zero. Hence
$$E_{0}=\ker\na\cap\Om^{1}(X)=0,$$
and therefore $H_{0}=\Om^{1}(X)$. We then get 
$$\la_{1}^{(1)}=\inf_{\a\neq 0} R(\a).$$

\end{proof}

	\subsection{A Poincar\'{e}-Sobolev Inequality under Curvature Bounds}
	In this section, we recall a Poincaré-type inequality (see \cite{Ber,Che}). For $1 \leq q$, let $p$ satisfy $1 \leq p \leq \frac{nq}{n-q}$ and $p < \infty$. Define the Sobolev constant $\Sigma(n,p,q)$ of the canonical unit sphere $S^n$ by
	\[
	\Sigma(n,p,q) = \sup\left\{ \frac{\|f\|_{L^p}}{\|df\|_{L^q}} : f \in L^q_1(S^n),\ f \neq 0,\ \int_{S^n} f = 0 \right\}.
	\]
	Suppose a closed Riemannian manifold $(X,g)$ satisfies
	\[
	\operatorname{Ric}(g) \geq -(n-1)\kappa \quad \text{and} \quad \operatorname{diam}(X) \leq D(g).
	\]
	Then there exists a positive number $R(\Lambda) = \frac{D(g)}{\Lambda C(\Lambda)}$, where $\Lambda = \sqrt{\kappa D^2}$ (see \cite[Appendix I, Theorem 2]{Ber}), and $C(\Lambda)$ is the unique positive root of the equation
	\[
	x \int_0^\Lambda (\cosh t + x \sinh t)^{n-1} dt = \int_0^\pi \sin^{n-1} t \, dt,
	\]
	such that the following Poincaré inequality holds (see \cite[Page 397]{Ber} or \cite[Theorem 5.1]{Che}).
	
	\begin{theorem}\label{T4}
		Let $(X,g)$ be a closed $n$-dimensional smooth Riemannian manifold. Then for each $1 \leq p \leq \frac{nq}{n-q}$, $p < \infty$, and $f \in L^q_1(X)$, we have
		\[
		\left\| f - \bar{f} \right\|_{L^p(X)} \leq S_{p,q} \|df\|_{L^q(X)},
		\]
		where $$\bar{f}= \frac{1}{\operatorname{Vol}(g)} \int_X f,$$ and 
		\[
		S_{p,q} = \left( \frac{\operatorname{Vol}(g)}{\operatorname{Vol}(S^n)} \right)^{\frac{1}{p} - \frac{1}{q}} R(\Lambda) \Sigma(n,p,q).
		\]
	\end{theorem}
		Chen \cite[Lemma 5.6]{Che} established a lower bound for $\Lambda C(\Lambda)$ as $\Lambda \to 0$. Following his approach, we derive a lower bound for $\Lambda C(\Lambda)$ that depends only on $\Lambda$ and $n$.
	
	\begin{lemma}\label{L2}
		Let $C(\Lambda)$ be the unique positive root of the equation
		\[
		x \int_0^\Lambda (\cosh t + x \sinh t)^{n-1} dt = \int_0^\pi \sin^{n-1} t \, dt.
		\]
		Then 
		\[
		\Lambda C(\Lambda) \geq a_n e^{-(n-1)\Lambda} > 0
		\]
		for some constant $a_n$ depending only on $n$.
	\end{lemma}
	
	\begin{proof}
		Let $\omega_n = \int_0^\pi \sin^{n-1} t \, dt$. Then
		\[
		\omega_n = C(\Lambda) \int_0^\Lambda (\cosh t + C(\Lambda) \sinh t)^{n-1} dt \geq \Lambda C(\Lambda).
		\]
		On the other hand, for any $\Lambda > 0$, we have
		\begin{align*}
		\omega_n &= C(\Lambda) \int_0^\Lambda (\cosh t + C(\Lambda) \sinh t)^{n-1} dt \\
		&\leq C(\Lambda) \int_0^\Lambda \left( \frac{e^\Lambda + e^{-\Lambda}}{2} + C(\Lambda) \frac{e^\Lambda - e^{-\Lambda}}{2} \right)^{n-1} dt \\
		&\leq \Lambda C(\Lambda) \cosh^{n-1}(\Lambda) \left( 1 + \omega_n \frac{\tanh(\Lambda)}{\Lambda} \right)^{n-1} \\
		&\leq \Lambda C(\Lambda) e^{(n-1)\Lambda} (1 + \omega_n)^{n-1}.
		\end{align*}
		Here we use the fact that $\frac{\tanh x}{x}$ is monotonically decreasing for $x > 0$, and that
		\[
		\sup_{x > 0} \frac{\tanh x}{x} = 1.
		\]
		Hence, for some constant $a_n = \omega_n (1 + \omega_n)^{1-n}$ depending only on $n$, we obtain
		\[
		\Lambda C(\Lambda) \geq a_n e^{-(n-1)\Lambda} > 0.
		\]
	\end{proof}
	
	\begin{proposition}\label{P1}
		Let $(X,g)$ be a closed $n$-dimensional smooth Riemannian manifold with
		\[
		\operatorname{Ric}(g) \geq -(n-1)\kappa \quad \text{and} \quad \operatorname{diam}(X) \leq D(g).
		\]
		If $f \in L^2_1(X)$, then
		\[
		\|f\|_{\frac{2n}{n-2}} \leq \|f\|_2 + C(n) D e^{(n-1)\sqrt{\kappa D^{2}}}  \|df\|_2.
		\]
	\end{proposition}
	\begin{proof}
	By Theorem \ref{T4} and Lemma \ref{L2}, we get
	\begin{equation*}
	\begin{split}
	\|f\|_{\frac{2n}{n-2}}&\leq\|f-\bar{f}\|_{\frac{2n}{n-2}}+\|\bar{f}\|_{\frac{2n}{n-2}}\\
	&\leq C(n)De^{(n-1)\sqrt{\kappa D^{2}}}+|\bar{f}|\\
	&\leq C(n)De^{(n-1)\sqrt{\kappa D^{2}}}+\|f\|_{2}.\\
	\end{split}
	\end{equation*}
		
	\end{proof}	
	Applying \cite[Appendix V, Theorem 3]{Ber}, we obtain the following result:
	
	\begin{theorem}\label{T5}(\cite[Theorem 5.2]{Che})
		Let $(X,g)$ be a closed $n$-dimensional smooth Riemannian manifold satisfying the Sobolev inequality:
		\[
		\|f\|_{\frac{2n}{n-2}} \leq \|f\|_2 + C_s \|df\|_2, \quad \text{for all } f \in L^2_1(X).
		\]
		If $u\in L^2_1(X)$ is a nonnegative continuous function satisfying
		\[
		u \Delta_d u \leq c u^2
		\]
		in the sense of distributions for some constant $c > 0$, then
		\[
		\|u\|_\infty \leq \exp\left\{ C(n) \sqrt{c} C_s \right\} \|u\|_2.
		\]
	\end{theorem}

\subsection{Anderson's Counterexample to Li-Yau Type Estimates}
In this section, we aim to demonstrate that there exist manifolds for which the lower bound estimates of the eigenvalues of the rough Laplacian acting on $1$-forms do not satisfy the Li-Yau type inequality.

We first recall the following key result.
\begin{lemma}(cf.\cite{Tis,Yam88})
	Let $(X, g)$ be a closed $n$-dimensional smooth Riemannian manifold such that
	\[
	\dim\left(\ker\nabla \cap \Omega^{1}(X)\right) = m > 0.
	\]
	Then $X$ is a fiber bundle over the $m$-torus.
\end{lemma}

\begin{proof}
	Let $\ker\nabla \cap \Omega^{1}(X) = \operatorname{span}\{\omega_1, \dots, \omega_m\}$. By an argument of Tischler \cite{Tis}, there exist differentiable functions $\varphi_i : X \to S^1$, $1 \leq i \leq m$, such that $\varphi_i^*(dt) = \omega_i$. The map $\varphi = (\varphi_1, \dots, \varphi_m)$ is a submersion from $X$ onto an $m$-torus $T^m$, and $X$ is a fiber bundle over $T^m$.
\end{proof}  

\begin{definition}\label{D1}
Let $(X, g)$ be a closed $n$-dimensional smooth Riemannian manifold with 
\[
\operatorname{Ric}(g)\geq -(n-1)\kappa,\ and\ {\rm{diam}}(X)\leq D.
\]
We say the first eigenvalue of the rough Laplacian on $1$-forms satisfies Li-Yau type inequality if there is a positive constant $C(n)$ such that 
$$\la^{(1)}_{1}D^2\geq C(n)e^{-C(n)\sqrt{\kappa D^2}}.$$
\end{definition}

We observe that if a closed manifold admits a metric such that $\kappa D^2$ is small and  $\la_{1}^{(1)}$ satisfies the Li-Yau type inequality, then the closed manifold is a fiber bundle over $T^{b_{1}}$.

\begin{proposition}\label{P5}
	Let $(X, g)$ be a closed $n$-dimensional smooth Riemannian manifold with $b_1(X) > 0$. Suppose that the Ricci curvature satisfies
	\[
	\operatorname{Ric}(g)\geq -(n-1)\kappa,\ and \ {\rm{diam}}(X)\leq D
	\]
	and the first eigenvalue of the rough Laplacian on $1$-forms satisfies Li-Yau type inequality. Then there is a positive constant $c(n)$ such that if $\kappa D^2\leq c(n)$, then we have
	\[
	\ker\nabla \cap \Omega^1(X) = \mathcal{H}_d^1(X).
	\]
	In particular, $X$ is a fiber bundle over a $b_1$-torus.
\end{proposition}

\begin{proof}
	Note that both $\ker\nabla \cap \Omega^1$ and $\mathcal{H}^1_d := \ker \Delta_d \cap \Omega^1$ are finite-dimensional linear spaces, and
	\[
	\ker\nabla \cap \Omega^1 \subset \mathcal{H}^1_d.
	\]
	We decompose
	\[
	\mathcal{H}^1_d = \left(\ker\nabla \cap \Omega^1\right) \oplus \left[ \left(\ker\nabla \cap \Omega^1\right)^\perp \cap \mathcal{H}^1_d \right].
	\]
	It suffices to show that the second summand is trivial.
	
	Suppose, for contradiction, that
	\[
	\left(\ker\nabla \cap \Omega^1\right)^\perp \cap \mathcal{H}^1_d \neq \{0\}.
	\]
	Then there exists a nonzero $1$-form $\alpha \in \mathcal{H}^1_d$ such that $\nabla \alpha \neq 0$. Write
	\[
	\alpha = \alpha_0 + \alpha_0^\perp,
	\]
	where $\alpha_0 \in \ker\nabla \cap \Omega^1$ and $\alpha_0^\perp \in \left(\ker\nabla \cap \Omega^1\right)^\perp$. Clearly, $\alpha_0^\perp \neq 0$. By the Bochner formula,
	\[
	0 = \Delta_d \alpha = \Delta_d \alpha_0^\perp = \nabla^*\nabla \alpha_0^\perp + \operatorname{Ric}(\alpha_0^\perp).
	\]
	Taking the inner product with $\alpha_0^\perp$ and integrating gives
	\[
	D^2\|\nabla \alpha_0^\perp\|^2 \leq (n-1)\kappa D^2  \|\alpha_0^\perp\|^2.
	\]
	On the other hand, the eigenvalue assumption implies
	\[
	D^2\|\nabla \alpha_0^\perp\|^2 \geq \lambda_1^{(1)}	D^2 \|\alpha_0^\perp\|^2 \geq C(n) e^{-C(n)\sqrt{\kappa D^2}} \|\alpha_0^\perp\|^2.
	\]
	Combining these inequalities yields
	\[
	\left( C(n) e^{-C(n)\sqrt{\kappa D^2}}- (n-1)\kappa D^2\right)  \|\alpha_0^\perp\|^2 \leq 0.
	\]
		Now, choosing $c(n)$ small enough ensures that if $\kappa D^2\leq c(n)$ then 
	\[
	C(n) e^{-C(n)\sqrt{\kappa D^2}}>(n-1)\kappa D^2,
	\]
	which forces $\alpha_0^\perp = 0$, a contradiction. Hence,
	\[
	\ker\nabla \cap \Omega^1 = \mathcal{H}^1_d.
	\]
\end{proof}
Andersen in \cite{An94} gave explicit counterexamples to Gromov's conjecture 
\begin{theorem}(\cite[Theorem 0.4]{An94})\label{T1}
	For any $n \geq 4$, $k \leq n-1$, and $\varepsilon > 0$, there exist closed $n$-manifolds $X$ satisfying
	\[
	\operatorname{diam}(X)^2 \cdot \|\mathrm{Ric}\|_{\infty} \leq \varepsilon, \quad \text{and} \quad b_1(X) = k,
	\]
	such that no cover of $X$ fibers over $S^1$.
\end{theorem}
\begin{proof}[\textbf{Proof of Proposition \ref{P2} }]
Through Proposition \ref{P5}, we can conclude that on the manifold constructed by Andersen,  the eigenvalue  $\la^{(1)}_{1}$ does not satisfy the Li-Yau type inequality. Otherwise, the manifold would necessarily be a fiber bundle over $T^{b_1}$, which contradicts Theorem \ref{T1}.
\end{proof}

\section{Manifolds with Integral Curvature Bounds}
\subsection{$L^{\infty}$ Estimates for Eigenforms of Rough Laplacian}

\begin{lemma}\label{L4}(cf. \cite[Lemma 4.1]{PS})
	Let $(X,g)$ be a closed $n$-dimensional smooth Riemannian manifold satisfying the Sobolev inequality:
	\[
	\|f\|_{\frac{2n}{n-2}} \leq \|f\|_{2} + C_{s}\|df\|_{2}, \quad \text{for all } f \in L^{2}_{1}(X).
	\]
	Suppose  $\theta$ is a $p$-form such that 
	$$\nabla^{*}\nabla\theta = \lambda\theta,$$
	where $\la$ is a positive constant. Then
	\begin{equation}\label{E1}
	\|\theta\|_{\infty} \leq \exp\left\{C(n)\sqrt{\lambda}C_{s}\right\} \|\theta\|_{2}.
	\end{equation}
\end{lemma}

\begin{proof}
	For any $\alpha \in \Omega^{p}(X)$, the following identity holds:
	\[
	-\frac{1}{2}\Delta_d|\alpha|^{2} = |\nabla\alpha|^{2} - \langle\nabla^{*}\nabla\alpha,\alpha\rangle.
	\]
	Applying this to $\theta$ and using Kato's inequality $|\nabla|\theta|| \leq |\nabla\theta|$, we obtain:
	\begin{align*}
	\frac{1}{2}\Delta_d|\theta|^{2} &= \langle\nabla^{*}\nabla\theta,\theta\rangle - |\nabla\theta|^{2} \\
	&\leq \lambda|\theta|^{2} - \left|\nabla|\theta|\right|^{2}.
	\end{align*}
	On the other hand, we have the pointwise identity:
	\[
	\frac{1}{2}\Delta_d|\theta|^{2} = |\theta| \cdot \Delta_d|\theta| + \left|\nabla|\theta|\right|^{2}.
	\]
	Combining these yields:
	\[
	|\theta| \cdot \Delta_d|\theta| \leq \lambda|\theta|^{2}.
	\]
	The estimate now follows directly from Theorem \ref{T5}.
\end{proof}

\begin{proposition}\label{P6}
	Let $n \geq 2$ and $2p > n$, and let $(X,g)$ be a closed $n$-manifold satisfying the Sobolev inequality:
	\[
	\|f\|_{\frac{2n}{n-2}} \leq \|f\|_{2} + C_{s}\|df\|_{2}, \quad \text{for all } f \in L^{2}_{1}(X).
	\]
	Suppose $\theta \in \Omega^{p}(X)$ is a $p$-form such that $$\nabla^{*}\nabla\theta = \lambda\theta$$
	where $\la$ is a positive constant. 
	
	Then there exists a constant $C(n,p) > 0$ such that
		\begin{equation}\label{E2}
		\begin{split}
	&D\|\nabla\theta\|_{\infty} \\
	&\leq\min\{C(n,p) (1 + \sqrt{t})^{\a} \sqrt{\lambda D^{2}} \|\theta\|_2,
	 C(n,p) (1 + \sqrt{t})^{\b} (\sqrt{\lambda D^{2}})^{\frac{\b}{\a} } \exp\{C(n,p)\sqrt{\la}C_{s} \} \|\theta\|_2\},\\
	 \end{split}
		\end{equation}
	where $\a=\frac{2pn}{2p-n}$ and $\b=\frac{2pn}{2p-n+pn}$,  $t = 4 C_s\sqrt{B}\sqrt{1+BD^{2}}$, $B=\lambda + \|\mathrm{Ric}^{-}\|_{p} + \|\mathrm{Riem}\|_{2p}$.
\end{proposition}

To prove Proposition \ref{P6}, we require a commutation lemma (see \cite{Aub,ACGR,LR}):

\begin{lemma}\label{L1}
	For any smooth section $S \in \Gamma(E)$, we have
	\begin{equation}
	\frac{1}{2}\Delta_d(|\na S|) + |\na^{2}S| \leq \langle \na^{\ast}R^{E}S, \na S\rangle + {\mathrm{Ric}}^{-}|\na S|^{2} + \langle \na\bar{\Delta}S, \na S\rangle + |R^{E}| \cdot |\na S|^{2},
	\end{equation}
	where for vector fields $U,V$ on $X$, 
	$$R^{E}_{U,V}=\na^{2}_{U,V}-\na^{2}_{V,U}$$
	is the curvature of $E$, and 
	$$\langle \na^{\ast}R^{E}S,\na S\rangle:=-\langle\na_{j}(R^{E}_{ij}S),\na_{i}S\rangle.$$

\end{lemma}
In our case, the fiber bundle $E$ is the cotangent bundle and $S\in \Omega^1(X)$ is a differential $1$-form,  so the bundle curvature  $R^E$  is identified with the Riemann curvature tensor $\mathrm{Riem}$.
We now present the proof of Proposition \ref{P6},  whose main technical idea originates from Aubry \cite[Proposition 4.2]{Aub}.
\begin{proof}[\textbf{Proof of Proposition \ref{P6}}]
For simplicity, we denote the Riemann curvature tensor $\mathrm{Riem}$ as $R$ throughout the proof. Let $u = |\nabla\theta|$.  We have the identity
	\[
	u(\Delta_d u) = \frac{1}{2}\Delta_d(u^2) + |\nabla u|^2.
	\]
	Hence, by Lemma \ref{L1}, we can get for any $k\geq1$,
	\begin{align*}
	\int_X |d(u^k)|^2 &= k^2 \int_X |\nabla u|^2 u^{2(k-1)} \\
	&= \frac{k^2}{2k-1} \int_X \langle u^{2k-1}, \Delta_d u \rangle \\
	&\leq \frac{k^2}{2k-1} \int_X \left( \frac{1}{2}\Delta_d(u^2) + |\nabla u|^2 \right) u^{2(k-1)} \\
	&\leq \frac{k^2}{2k-1} \left( \int_X \langle {\mathrm{Ric}}^{-} u^{2k} + \int_X \langle \nabla\bar{\Delta}\theta, \nabla\theta \rangle u^{2(k-1)} \right. \\
	&\quad + \left. \int_X \langle \nabla^* R\theta, \nabla\theta \rangle u^{2(k-1)} + \int_X |R| u^{2k} \right) \\
	&\leq \frac{k^2}{2k-1} \left( \int_X \langle \mathrm{Ric}^{-} u^{2k} + \int_X (\lambda + |R|) u^{2k}+\int_X \langle \nabla^* R\theta, \nabla\theta \rangle u^{2(k-1)} \right).
	\end{align*}
We now estimeate the last term $\int_X \langle \nabla^* R\theta, \nabla\theta \rangle u^{2(k-1)} $.  Using the relation 
$$|ab|\leq |a|^2 + \frac{|b|^2}{4}\leq |a|^2 +\frac{|b|^2}{2}$$ and the divergence theorem,  we observe that
	\begin{align*}
	\int_X \langle \nabla^* R\theta, \nabla\theta \rangle u^{2(k-1)} 
	&= \int_X \langle \nabla^* R\theta, u^{2(k-1)} \nabla\theta \rangle \\
	&= 2(k-1) \int_X \langle R\theta, (u^{2k-3} \nabla u) \nabla\theta \rangle + \int_X \langle R\theta, (\nabla^2 \theta) u^{2(k-1)} \rangle \\
	&\leq \int_X |R\theta|^2 \cdot u^{2(k-1)} + 2(k-1) \int_X |R\theta| \cdot |\nabla u| \cdot u^{2(k-1)} \\
	&\leq \int_X |R\theta|^2 \cdot u^{2(k-1)} + 2(k-1) \left( \int_X |R\theta|^2 \cdot u^{2(k-1)} + \frac{1}{2} \int_X |\nabla u|^2 u^{2(k-1)} \right) \\
	&\leq (2k-1) \int_X |R\theta|^2 \cdot u^{2(k-1)} + (k-1) \int_X |\nabla u|^2 u^{2(k-1)}.
	\end{align*}
From this estimate, it follows that
	\begin{align*}
	\int_X |d(u^k)|^2 &\leq \frac{k^2}{2k-1} \left( \int_X (\lambda + \mathrm{Ric}^{-} ) u^{2k} + \int_X |R| u^{2k} \right. \\
	&\quad + \left. (2k-1) \int_X |R\theta|^2 \cdot u^{2(k-1)} + \frac{k-1}{k^2} \int_X |d(u^k)|^2 \right).
	\end{align*}	
After rearranging terms, we use  H\"older’s inequality to obtain
	\begin{align*}
	\int_X |d(u^k)|^2 &\leq k \left( \int_X (\lambda + \mathrm{Ric}^{-}+|R| ) u^{2k} + (2k-1)k \int_X |R\theta|^2 \cdot u^{2(k-1)} \right) \\
	&\leq  \left( k(\lambda + \|\mathrm{Ric}^{-} \|_p + \|R\|_p) \|u\|^{2k}_{\frac{2kp}{p-1}} + 4k^{2}\|R\|^2_{2p} \|\theta\|^2_\infty \|u\|^{2(k-1)}_{\frac{2(k-1)p}{p-1}} \right) \\
	&\leq 4k^{2}B\|u\|^{2k}_{\frac{2kp}{p-1}} + 4k^{2}B^{2}\|\theta\|^2_\infty \|u\|^{2(k-1)}_{\frac{2(k-1)p}{p-1}},
	\end{align*}
	where we define
	\begin{align*}
	B_1 &= \lambda + \|\mathrm{Ric}^{-}\|_p + \|R\|_{p} \leq \lambda + \|\mathrm{Ric}^{-}\|_{p}  + \|R\|_{2p} := B, \\
	B_2 &= \|R\|^2_{2p} \leq B^2.
	\end{align*}
Combing the Sobolev inequality in Theorem \ref{T5}  and the estimate for $\int_X |d(u^k)|^2$, we conclude that 
	\begin{align*}
	\|u\|^k_{\frac{2nk}{n-2}} &= \|u^k\|_{\frac{2n}{n-2}} \leq \|u^k\|_2 + C_s \|d(u^k)\|_2 \\
	&\leq \|u^k\|_2 + 2C_s k  \sqrt{ B\|u\|^{2k}_{\frac{2kp}{p-1}} +B^{2} \|\theta\|^2_\infty \|u\|^{2(k-1)}_{\frac{2(k-1)p}{p-1}} } \\
	&\leq \left( \|u\|_\infty^{\frac{1}{k}} \|u\|^{1-\frac{1}{k}}_{\frac{2(k-1)p}{p-1}} \right)^k \\
	&\quad + 2C_s k \sqrt{ B\left( \|u\|_\infty^{\frac{1}{k}} \|u\|^{1-\frac{1}{k}}_{\frac{2(k-1)p}{p-1}} \right)^{2k} + B^{2}\|\theta\|^2_\infty \|u\|^{2(k-1)}_{\frac{2(k-1)p}{p-1}} } \\
	&= \left( \|u\|_\infty + 2C_s k \sqrt{ B\|u\|^2_\infty +B^{2} \|\theta\|^2_\infty } \right) \|u\|^{k-1}_{\frac{2(k-1)p}{p-1}},
	\end{align*}
	where we used the interpolation inequality
	\[
	\|u\|_{2k} \leq \|u\|_{\frac{2kp}{p-1}} \leq \|u\|^{\frac{1}{k}}_\infty \|u\|^{1-\frac{1}{k}}_{\frac{2(k-1)p}{p-1}}.
	\]
		Therefore,
	\[
	\left( \frac{\|u\|_{\frac{2nk}{n-2}}}{\|u\|_\infty} \right)^k
	\leq \left( 1 + 2C_s k\sqrt{B}\sqrt{ 1 + \frac{B \|\theta\|^2_\infty}{\|u\|^2_\infty} } \right)
	\left( \frac{\|u\|_{\frac{2(k-1)p}{p-1}}}{\|u\|_\infty} \right)^{k-1}.
	\]	
Taking both sides to the power of $\frac{2n}{n-2}$ yields
	\begin{align*}
	\left( \frac{\|u\|_{\frac{2nk}{n-2}}}{\|u\|_\infty} \right)^{\frac{2nk}{n-2}}
	&\leq \left( 1 + 2C_s k\sqrt{B} \sqrt{ 1 + \frac{B \|\theta\|^2_\infty}{\|u\|^2_\infty} } \right)^{\frac{2n}{n-2}}
	\left( \frac{\|u\|_{\frac{2(k-1)p}{p-1}}}{\|u\|_\infty} \right)^{\frac{2n(k-1)}{n-2}} \\
	&= \left( 1 +2C_s k\sqrt{B}  \sqrt{ 1 + \frac{B \|\theta\|^2_\infty}{\|u\|^2_\infty} } \right)^{\frac{2n}{n-2}}
	\left( \frac{\|u\|_{\frac{2(k-1)p}{p-1}}}{\|u\|_\infty} \right)^{\gamma \frac{2(k-1)p}{p-1}},
	\end{align*}
	where $\gamma = \frac{n(p-1)}{p(n-2)} > 1$.
	
	Define the sequence $\{a_i\}_{i\geq 0}$ by $a_0 = \frac{2p}{p-1}$ and
	\[
	a_{i+1} = \gamma a_i + \frac{2n}{n-2}.
	\]
	Then
	\[
	a_i = \gamma^i (a_0 + \gamma_0) - \gamma_0,
	\]
	where $\gamma_0 = \frac{2n}{(n-2)(\gamma-1)} = \frac{2np}{2p-n}$. Let
	\[
	k_i = \frac{p-1}{2p} a_i + 1 = \gamma^{i+1} + \gamma^i - \gamma + 1 < 2\gamma^{i+1},
	\]
	which satisfies
	\[
	\frac{2n}{n-2} k_i = \frac{n(p-1)}{p(n-2)} a_i + \frac{2n}{n-2} = a_{i+1}.
	\]
	Therefore, {substituting $k=k_i$ gives}
	\begin{align*}
	\left( \frac{\|u\|_{a_{i+1}}}{\|u\|_\infty} \right)^{\frac{a_{i+1}}{\gamma^{i+1}}}
	&\leq \left( 1 + 2C_s k_{i}\sqrt{B} \sqrt{ 1 + \frac{B \|\theta\|^2_\infty}{\|u\|^2_\infty} } \right)^{\frac{2n}{(n-2)\gamma^{i+1}}}
	\left( \frac{\|u\|_{a_i}}{\|u\|_\infty} \right)^{\frac{a_i}{\gamma^i}} \\
	&\leq \prod_{l=0}^i \left( 1 + 2C_s  k_l\sqrt{B}\sqrt{ 1 + \frac{B \|\theta\|^2_\infty}{\|u\|^2_\infty} } \right)^{\frac{2n}{(n-2)\gamma^{l+1}}}
	\left( \frac{\|u\|_{a_0}}{\|u\|_\infty} \right)^{a_0}.
	\end{align*}
	To estimate the infinite product, we need the following lemma.
	\begin{lemma}\label{L6}
		Let $t > 0$, $\gamma > 1$. Then
		\[
		P = \prod_{i=0}^\infty (1 + t\gamma^{i+1})^{\frac{1}{\gamma^{i+1}}}
		\leq \exp\left\{ \frac{2\sqrt{\gamma}}{\gamma-1} \right\} (1 + \sqrt{t})^{\frac{2}{\gamma-1}}.
		\]
	\end{lemma}
	
	\begin{proof}
		For any $x > 0$, we have
		\begin{align*}
		\ln(1 + tx) &\leq 2\ln(1 + \sqrt{tx}) \\
		&= 2\ln(1 + \sqrt{t}) + 2\ln\left( 1 + (\sqrt{x} - 1)\frac{\sqrt{t}}{1 + \sqrt{t}} \right) \\
		&\leq 2\ln(1 + \sqrt{t}) + 2(\sqrt{x} - 1).
		\end{align*}
		Hence,
		\begin{align*}
		\ln P &= \sum_{i=0}^\infty \frac{\ln(1 + t\gamma^{i+1})}{\gamma^{i+1}} \\
		&\leq \sum_{i=0}^\infty \frac{2\ln(1 + \sqrt{t})}{\gamma^{i+1}} + \sum_{i=0}^\infty \frac{2\gamma^{\frac{i+1}{2}} - 2}{\gamma^{i+1}} \\
		&\leq \frac{2\ln(1 + \sqrt{t})}{\gamma-1} + \frac{2}{\sqrt{\gamma} - 1} - \frac{2}{\gamma-1} \\
		&= \frac{2\ln(1 + \sqrt{t})}{\gamma-1} + \frac{2\sqrt{\gamma}}{\gamma-1}.
		\end{align*}
		Therefore, we complete the proof of  this lemma.
	\end{proof}
	{Now we continue to prove Proposition \ref{P6}.} Taking the limit as $i \to \infty$, we obtain
	\begin{align*}
	1 &= \lim_{i\to\infty} \left( \frac{\|u\|_{a_{i+1}}}{\|u\|_\infty} \right)^{\frac{a_{i+1}}{\gamma^{i+1}}} \\
	&\leq \prod_{l=0}^\infty \left( 1 + 2C_s  k_l\sqrt{B} \sqrt{ 1 + \frac{B \|\theta\|^2_\infty}{\|u\|^2_\infty} } \right)^{\frac{2n}{(n-2)\gamma^{l+1}}}
	\left( \frac{\|u\|_{a_0}}{\|u\|_\infty} \right)^{a_0} \\
	&\leq \prod_{l=0}^\infty \left( 1 + 2C_s k_l\sqrt{B} \sqrt{ 1 + \frac{B \|\theta\|^2_\infty}{\|u\|^2_\infty} } \right)^{\frac{2n}{(n-2)\gamma^{l+1}}}
	\left( \frac{\|u\|^{1-\frac{1}{p}}_2 \|u\|^{\frac{1}{p}}_\infty}{\|u\|_\infty} \right)^{a_0} \\
	&= \prod_{l=0}^\infty \left( 1 + 2C_s k_l\sqrt{B} \sqrt{ 1 + \frac{ B\|\theta\|^2_\infty}{\|u\|^2_\infty} } \right)^{\frac{2n}{(n-2)\gamma^{l+1}}}
	\left( \frac{\|u\|^2_2}{\|u\|^2_\infty} \right) \\
	&\leq \prod_{l=0}^\infty \left( 1 + 4C_s \gamma^{l+1}\sqrt{B} \sqrt{ 1 + \frac{B \|\theta\|^2_\infty}{\|u\|^2_\infty} } \right)^{\frac{2n}{(n-2)\gamma^{l+1}}}
	\left( \frac{\|u\|^2_2}{\|u\|^2_\infty} \right),
	\end{align*}
	where we used the interpolation
	\[
	\|u\|_{a_0} \leq \|u\|^{1-\frac{1}{p}}_2 \|u\|^{\frac{1}{p}}_\infty.
	\]
Therefore,
	\[
	\|u\|^2_\infty \leq \prod_{l=0}^\infty \left( 1 + 4C_s\gamma^{l+1}\sqrt{B} \sqrt{ 1 + \frac{B \|\theta\|^2_\infty}{\|u\|^2_\infty} } \right)^{\frac{2n}{(n-2)\gamma^{l+1}}} \|u\|^2_2.
	\]
	We now consider two cases:\\
	\textbf{Case 1:} If $D^{2}\|u\|^{2}_\infty \geq \|\theta\|^{2}_\infty$, then
	\begin{align*}
	\|u\|^2_\infty &\leq \prod_{l=0}^\infty \left( 1 + 4 C_s\sqrt{B}\sqrt{1+BD^{2}} \gamma^{l+1} \right)^{\frac{2n}{(n-2)\gamma^{l+1}}} \|u\|^2_2.
	\end{align*}
	Let $t = 4 C_s\sqrt{B}\sqrt{1+BD^{2}}$. By Lemma \ref{L6}, we obtain
	\begin{align*}
	\|u\|^2_\infty &\leq C(n,p) (1 + \sqrt{t})^{\frac{4pn}{2p-n}} \|u\|^2_2 \\
	&= C(n,p) (1 + \sqrt{t})^{\frac{4pn}{2p-n}} \lambda \|\theta\|^2_2.
	\end{align*}
\textbf{Case 2:} If $D^{2}\|u\|^{2}_\infty \leq \|\theta\|^{2}_\infty$, then
	\begin{align*}
	\|u\|^2_\infty \|u\|_\infty^{\frac{2pn}{2p-n}} 
	&\leq \prod_{l=0}^\infty \left( \|u\|_\infty + 4C_s \sqrt{B} \gamma^{l+1} \sqrt{ \|u\|^2_\infty + B\|\theta\|^2_\infty } \right)^{\frac{2n}{(n-2)\gamma^{l+1}}} \|u\|^2_2 \\
	&\leq \prod_{l=0}^\infty \left( \frac{\|u\|_\infty}{\|\theta\|_\infty} + 4C_s \sqrt{B} \gamma^{l+1} \sqrt{ \frac{\|u\|^2_\infty}{\|\theta\|^2_\infty} + B } \right)^{\frac{2n}{(n-2)\gamma^{l+1}}} \|u\|^2_2 \|\theta\|^{\frac{2pn}{2p-n}}_\infty \\
	&{\leq \prod_{l=0}^\infty \left(\frac{1}{D}+ 4 C_s \sqrt{B}\sqrt{\frac{1}{D^2}+B} \gamma^{l+1} \right)^{\frac{2n}{(n-2)\gamma^{l+1}}} \|u\|^2_2 \|\theta\|^{\frac{2pn}{2p-n}}_\infty} \\
	&\leq C(n,p) (1+\sqrt{t})^{\frac{4pn}{2p-n}}D^{-\frac{2pn}{2p-n}} \|u\|^2_2 \|\theta\|^{\frac{2pn}{2p-n}}_\infty \\
	&= C(n,p) (1 + \sqrt{t})^{\frac{4pn}{2p-n}} D^{-\frac{2pn}{2p-n}}\lambda \|\theta\|^2_2 \|\theta\|^{\frac{2pn}{2p-n}}_\infty\\
	&\leq C(n,p) (1 + \sqrt{t})^{\frac{4pn}{2p-n}} D^{-\frac{2pn}{2p-n}}\lambda \exp\{C(n)\frac{2pn}{2p-n}\sqrt{\la}C_{s} \} \|\theta\|^{2+\frac{2pn}{2p-n} }_2.
	\end{align*}
Combining the above two cases, we arrive at the inequality estimate (\ref{E2}).
\end{proof}

\subsection{A Non-Vanishing Criterion for Small Eigenvalues}

We begin by recalling an inequality for smooth functions on closed manifolds.
We denote by
$$\varepsilon=\min\{C(n,p) (1 + \sqrt{t})^{\a} \sqrt{\lambda D^{2}},
C(n,p) (1 + \sqrt{t})^{\b} (\sqrt{\lambda D^{2}})^{\frac{\b}{\a} } \exp\{C(n,p)\sqrt{\la}C_{s} \}\}.$$
	where $\a=\frac{2pn}{2p-n}$ and $\b=\frac{2pn}{2p-n+pn}$,  $t = 4 C_s\sqrt{B}\sqrt{1+BD^{2}}$, $B=\lambda + \|\mathrm{Ric}^{-}\|_{p} + \|\mathrm{Riem}\|_{2p}$.
\begin{theorem}\label{T6}(cf.\cite{Heb})
	Let $(X,g)$ be a connected, closed Riemannian manifold with diameter $D(g)$, and let $f:X\rightarrow\mathbb{R}$ be a smooth function. Then for any $p,q\in X$,
	\begin{equation}
	|f(p)-f(q)| \leq \|\nabla f\|_{\infty} \cdot D(g).
	\end{equation}
\end{theorem}

Using this, we can show that when the eigenvalue is sufficiently small, the corresponding eigen $p$-form is nowhere vanishing.

\begin{theorem}\label{T2}
	Let $(X,g)$ be a closed $n$-dimensional smooth Riemannian manifold satisfying the Sobolev inequality:
	\[
	\|f\|_{\frac{2n}{n-2}} \leq \|f\|_{2} + C_{s}\|df\|_{2}, \quad \text{for all } f \in L^{2}_{1}(X).
	\]
	Suppose $\theta$ is a $p$-form such that
	$$\nabla^{\ast}\nabla\theta = \lambda\theta,$$
	where  $\la$ is a positive constant. Then
	\begin{equation*}
	\frac{\inf|\theta|^{2}}{\sup|\theta|^{2}} \geq 1 -2\varepsilon
	\end{equation*}
\end{theorem}

\begin{proof}
	Let $f = |\theta|^2$. Since $\nabla f = 2|\theta|\nabla|\theta|$ and by Kato's inequality $|\nabla|\theta|| \leq |\nabla\theta|$, we have
	\[
	|\nabla f| \leq 2|\theta||\nabla\theta|.
	\]
	Theorem \ref{T6} then implies
	\[
	\sup f - \inf f \leq 2\|\nabla\theta\|_{\infty} \cdot \|\theta\|_{\infty} \cdot D(g).
	\]
By Proposition \ref{P6}  and the fact that $\sup f\geq\|\theta\|^{2}_{2}$, we finally obtain
	 \begin{align*}
	\inf f \geq \sup f(1-2D\|\nabla\theta\|_{\infty} )\geq\sup f(1-2\varepsilon) .
	\end{align*}	
\end{proof}
Recall that the Gauss–Bonnet–Chern theorem is closely related to the Poincaré–Hopf index formula, which expresses the Euler characteristic as the sum of indices of singularities of a tangent vector field. As a consequence, if the eigenvalue on $1$-forms is sufficiently small, then the Euler characteristic must vanish.
\begin{proof}[\textbf{Proof of Theorem \ref{T3}}]
	By Proposition \ref{P1}, we may take the Sobolev constant to be 
	\[
	C_{s} = C(n) e^{(2n-1)\sqrt{\kappa D^{2}}} D.
	\]
Recall that  $B:=\lambda + \|\mathrm{Ric}^{-}\|_{p}  + \|\mathrm{Riem}\|_{2p}$ and $t = 4 C_s\sqrt{B}\sqrt{1+BD^{2}}$.

Since $\operatorname{Ric}(g) \geq -(2n-1)\kappa$, it follows that if 
$$\sqrt{\lambda D^2}\leq e^{-(2n-1)\sqrt{\kappa D^{2}}},$$ 
then  there exists a constant $C_0(n,p)>0$ such that 
	\begin{equation*}
	\begin{split}
	1+\sqrt{t}&=1+\big{(}4C(n)e^{(2n-1)\sqrt{k D^{2}}}\sqrt{BD^{2}}\sqrt{1+BD^{2}}\big{)}^{\frac{1}{2}}\\
	&\leq 1+\sqrt{4C(n)e^{(2n-1)\sqrt{k D^{2}}}}\sqrt{1+BD^{2}}.\\
	&\leq C_0(n,p)e^{(2n-1)\sqrt{k D^{2}}}\left(1+\sqrt{\|\mathrm{Riem}\|_{2p}}\,D\right)
	\end{split}
	\end{equation*}	
Take
$$\tilde{C}(n,p)=\frac{1}{(4C(n,p))^{\frac{1}{\de}}C_0(n,p)}\exp\left(-\frac{C(n)}{\de}\right),$$
where $\de=\frac{2pn}{p-n}$ or $\frac{2pn}{p-n+pn}$ (note that the $\dim X=2n$).

	Suppose, for contradiction, that $\lambda = \lambda^{(1)}_1$ is so small that
	\[
	\sqrt{\lambda D^2} \leq \min\left\{ \left( \frac{\tilde{C}(n,p)}{1 + \sqrt{\|\mathrm{Riem}\|_{2p}}\,D}e^{-(2n-1)\sqrt{\kappa D^{2}}} \right)^{\frac{2pn}{p-n}}, e^{-(2n-1)\sqrt{\kappa D^{2}}} \right\},
	\]
We can estimate the following term in the inequality in Theorem \ref{T2}.
\begin{equation*}
\begin{split}
\varepsilon&=C(n,p)(1+\sqrt{t})^{\de }(\sqrt{\la D^{2}})^{\frac{p-n}{2pn}\de}\exp\{C(n,p)\sqrt{\la}C_{s} \}\\
&\leq C(n,p)\left(C_0(n,p)e^{(2n-1)\sqrt{k D^{2}}}\left(1+\sqrt{\|\mathrm{Riem}\|_{2p}}\,D\right)\right)^{\de}\\
&\times \left( \frac{\tilde{C}(n,p)}{1 + \sqrt{\|\mathrm{Riem}\|_{2p}}\,D}e^{-(2n-1)\sqrt{\kappa D^{2}}} \right)^{\de}\exp\{C(n)\}\\
&\leq \frac{1}{4}.\\
\end{split}
\end{equation*}
Therefore, Theorem \ref{T2} implies
	\[
	\frac{\inf |\theta|^{2}}{\sup |\theta|^{2}} \geq \frac{1}{2}.
	\]
In particular, there exists a nowhere-zero $1$-form, and hence the Euler characteristic of $X$ must vanish, contradicting the hypothesis $\chi(X) \neq 0$.

\end{proof}

\section{ Geometric and Topological Applications}

\subsection{Harmonic $1$-Forms and Vanishing Euler Characteristic}
In this section, we first establish the connection between the first eigenvalue of the rough Laplacian acting on $1$-forms and the lower bound of the Ricci curvature.

\begin{proposition}\label{P3}
	Let $(X, g)$ be a closed $n$-dimensional smooth Riemannian manifold. Suppose the Ricci curvature of $X$ satisfies
	\[
	\operatorname{Ric}(g) \geq -\kappa,
	\]
	where $\kappa > 0$. If $b_1(X) > 0$, then {either $\ker \nabla \cap \Omega^1(X) \neq \{0\}$} or 
	\[
	\lambda_1^{(1)} \leq\kappa.
	\]
\end{proposition}

\begin{proof}
	{Since $b_1(X) > 0$, there exists a nonzero harmonic $1$-form $\alpha$ on $X$}. By the Weitzenb\"ock formula,
	\[
	0 = \Delta_d \alpha = \nabla^* \nabla \alpha + \operatorname{Ric}(\alpha).
	\]
	{Therefore, if $\ker \nabla \cap \Omega^1(X) =\{0\}$}, then by Lemma \ref{L5} we have
	\[
	\lambda_1^{(1)} \leq R(\alpha) = \frac{\|\nabla \alpha\|^2}{\|\alpha\|^2} = \frac{-\int_X \operatorname{Ric}(\alpha, \alpha)}{\|\alpha\|^2} \leq\kappa.
	\]
	
\end{proof}

\begin{corollary}\label{C3}
	Let $X$ be a closed $2n$-dimensional Riemannian manifold with $b_1(X) > 0$. Suppose there exists a sequence of metrics $\{g_i\}$ on $X$ such that
	\[
	\operatorname{Ric}(g_i) \geq -\frac{1}{i}, \quad \|\operatorname{Riem}(g_i)\|_{2p} \leq K, \quad \text{and} \quad \operatorname{diam}(X) \leq 1.
	\]
	Then the Euler characteristic of $X$ satisfies $\chi(X) = 0$.
\end{corollary}

\begin{proof}
	Suppose, for contradiction, that $\chi(X) \neq 0$. Then $\ker \nabla \cap \Omega^1(X) = \{0\}$. For each metric $g_i$, Proposition \ref{P3} implies that the corresponding first eigenvalue of the rough Laplacian on $1$-forms satisfies
	\[
	\lambda_1^{(1)}(g_i) \leq \frac{1}{i}.
	\]
	On the other hand, for sufficiently large $i$, 
	$$e^{-(2n-1)\sqrt{\frac{D^{2}}{(2n-1)i}}}\geq\frac{1}{2}$$.
	Therefore, Theorem \ref{T3} gives
	\[
\sqrt{\lambda_1^{(1)}(g_i)} \geq \min\left\{ \frac{1}{2}\left( \frac{\tilde{C}(n,p)}{1 + \sqrt{\|\operatorname{Riem}\|_{2p}}} \right)^{\frac{2pn}{p-n}}, \frac{1}{2}\tilde{C}(n,p) \right\}.
	\]
	This leads to a contradiction. Hence, $\chi(X) = 0$.
\end{proof}

\subsection{Killing Vector Fields and Finiteness of the Isometry Group}
When the isometry group of a closed Riemannian manifold $(X, g)$ is infinite, one can also relate the first eigenvalue of the rough Laplacian on $1$-forms to an upper bound on the Ricci curvature.

\begin{proposition}\label{P4}
	Let $(X, g)$ be a closed $n$-dimensional smooth Riemannian manifold. Suppose the Ricci curvature satisfies
	\[
	\operatorname{Ric}(g) \leq\kappa,
	\]
	for some constant $\kappa > 0$. If the isometry group of $g$ is infinite, then either $\ker \nabla \cap \Omega^1(X) \neq \{0\}$ or
	\[
	\lambda_1^{(1)} \leq\kappa.
	\]
\end{proposition}

\begin{proof}
	Since $\operatorname{Iso}(X, g)$ is infinite, there exists a nonzero Killing vector field $V$ on $X$ (i.e., $\mathcal{L}_V g = 0$). 
	By \cite[Proposition 1.4]{Pet}, we have the identity
	\[
	-\frac{1}{2} \Delta_d |V|^2 = |\nabla V|^2 - \operatorname{Ric}(V, V).
	\]
	Integrating over $X$ yields
	\[
	\int_X |\nabla V|^2 - \int_X \operatorname{Ric}(V, V) = 0.
	\]
	Let $\alpha = V^\flat$ be the $1$-form dual to $V$. Note that $|\nabla \alpha|^2 = |\nabla V|^2$ and $|\alpha|^2 = |V|^2$.  If $\ker \nabla \cap \Omega^1(X) = \{0\}$, then by Lemma~\ref{L5} we obtain
	\[
	\lambda_1^{(1)} \leq R(\alpha) = \frac{\|\nabla \alpha\|_{L^2}^2}{\|\alpha\|_{L^2}^2} = \frac{\|\nabla V\|_{L^2}^2}{\|V\|_{L^2}^2}.
	\]
	Using the integrated identity above, it follows that
	\[
	\lambda_1^{(1)} \leq \frac{\int_X \operatorname{Ric}(V, V)}{\|V\|_{L^2}^2} \leq \kappa.
	\]
\end{proof}

\begin{proof}[\textbf{Proof of Corollary \ref{C4}}]
	Suppose, for contradiction, that $\chi(X) \neq 0$. Then $\ker \nabla \cap \Omega^1 = 0$. 
	By Proposition~\ref{P4}, the first eigenvalue of the rough Laplacian on $1$-forms satisfies $\lambda_1^{(1)} \leq \lambda$. 
	This leads to a contradiction (see Theorem \ref{T3}), and hence $\chi(X) = 0$.
\end{proof}

  \section*{Acknowledgements}
We are very grateful to Prof. Aubry for the detailed explanation of the specifics in his articles \cite{Aub,ACGR}. This work is supported by the National Natural Science Foundation of China Nos. 12271496 and the Youth Innovation Promotion Association CAS, the Fundamental Research Funds of the Central Universities, and the USTC Research Funds of the Double First-Class Initiative.
	
\noindent\textbf{Data availability} {This manuscript has no associated data.}
	\section*{Declarations}
\noindent\textbf{Conflict of interest} The author states that there is no conflict of interest.

\end{document}